\documentclass[letterpaper,10pt]{amsart}
\usepackage{indentfirst} 
\usepackage{amssymb}
\usepackage{mathrsfs}
\usepackage{amsmath}
\usepackage{mathabx} 
\usepackage{amsthm}
\usepackage{thmtools}

\usepackage{bbm}             
\usepackage[colorlinks=true]{hyperref}  
\usepackage[usenames,dvipsnames]{xcolor} 
\usepackage{tikz}

\theoremstyle{plain}
\newtheorem{theorem}{Theorem}
\newtheorem{proposition}{Proposition}
\newtheorem{lemma}{Lemma}

\theoremstyle{definition}

\newtheorem{definition}{Definition}

\theoremstyle{remark}
\newtheorem{remark}{Remark}
\declaretheorem{claim}
\declaretheorem[name=Acknowledgements,numbered=no]{ack}
\declaretheorem[name=Outline of the paper,numbered=no]{outline}

\newcommand{\G}{\mathcal{G}}

\newcommand{\R}{\mathbb{R}}
\newcommand{\Z}{\mathbb{Z}}

\def\phi{\varphi}
\def\R{{\mathbb R}}

\def\Z{{\mathbb Z}}

\def\T{{\mathbb T}}

\begin{document}

\title[Gevrey regularity for the Vlasov-Poisson system]{Gevrey regularity for the Vlasov-Poisson system}
\date{\today}

\author[Renato Velozo Ruiz]{Renato Velozo Ruiz} \address{Centre for Mathematical Sciences,
University of Cambridge, Wilberforce Road, Cambridge, CB3 0WB, UK.}
\email{rav25@cam.ac.uk}

\begin{abstract} 
We prove propagation of $\frac{1}{s}$-Gevrey regularity $(s\in(0,1))$ for the Vlasov-Poisson system on $\T^d$ using a Fourier space method in analogy to the results proved for the 2D Euler system in \cite{KV} and \cite{LO}. More precisely, we give a quantitative estimate for the growth in time of the $\frac{1}{s}$-Gevrey norm for the solution of the system in terms of the force field and the gradient in the velocity variable of the distribution of matter. As an application, we show global existence of $\frac{1}{s}$-Gevrey solutions ($s\in (0,1)$) for the Vlasov-Poisson system in $\T^3$. Furthermore, the propagation of Gevrey regularity can be easily modified to obtain the same result in $\R^d$. In particular, this implies global existence of analytic $(s=1)$ and $\frac{1}{s}$-Gevrey solutions ($s\in (0,1)$) for the Vlasov-Poisson system in $\R^3$.
\end{abstract}

\maketitle

\section{Introduction} \label{introduction}
\subsection{Historical context}
Throughout this paper, we study the propagation of high order regularity for solutions $f(t,x,v)$ of the \emph{Vlasov-Poisson system} given by 
\begin{equation}\label{vlasov-poisson}
\begin{cases}
\partial_t f+v\cdot \nabla_xf+F(t,x) \cdot \nabla_vf=0,\\ F(t,x):=-\nabla_x W*_x\Big(\rho_f(t,x)-\dfrac{1}{(2\pi)^d}\|f_0\|_{L^1_{x,v}}\Big) \\
\rho_f(t,x):=\int_{\R^d} f(t,x,v) dv, \\
 f(t=0,x,v)=f_0(x,v),
\end{cases}
\end{equation}
where $t\in \R$, $x\in \T^d$ (or $x\in \R^d$), $v\in \R^d$ and $W$ denotes the usual Coulomb or Newtonian potential. 

This nonlinear transport equation is a well known model in astrophysics and plasma physics for the description of collisionless systems, see for instance \cite{BT} and \cite{LPit} .

The study of classical solutions for this system has been the topic of several papers written down since the beginning of the nineties. In particular, the problem of finding global classical solutions in dimension three has been studied in several works, see \cite{BD}, \cite{BR}, \cite{LP}, \cite{Pf} and \cite{S}. The first positive answer to this problem was given by Bardos and Degond for small initial data in \cite{BD}. Later on, Pfaffelmoser proved a global existence theorem for compactly supported initial data in $\R^3$, see \cite{Pf}. This result was later extended to $\T^3$ by Batt and Rein, see \cite{BR}. Afterwards, Lions and Perthame proved a result about propagation of moments which they used to show another global existence theorem in $\R^3$, but avoiding the compact support assumption in the initial data considered by Pfaffelmoser, see \cite{LP}. 

Around the same time when Pfaffelmoser and Lions-Perthame showed the global existence of classical solutions, Benachour proved a local existence result for analytic solutions on $\R^n$, see \cite{Be}. This result was motivated by the works \cite{Ba}, \cite{BB} and \cite{BBZ} about the global existence of analytic solutions for the 2D Euler system on $\R^2$, bounded domains and $\T^2$ respectively. Consequently, \cite{Be} leaves open the question of global existence of analytic solutions for the Vlasov-Poisson system in dimension three, as another of the several well-known analogies between this system and the 2D Euler system. 

Later on, Levermore and Oliver proved the propagation of analytic regularity for the 2D Euler system on $\T^2$ using a Fourier method in \cite{LO}, which comes from the notion of Gevrey regularity introduced by Gevrey in \cite{G}. Subsequently, the result in \cite{LO} was improved obtaining better estimates by Kukavica and Vicol in \cite{KV}.

The Gevrey regularity is a very strong notion of regularity given to functions which are $C^{\infty}$ and close to be analytic in the sense that they have stretched exponential decay of amplitudes of oscillations along frequencies, see Definition \ref{gevrey} and Section \ref{preliminaries} for further details. 

This notion has been proved to be useful in many different settings in partial differential equations, see for instance \cite{BMM}, \cite{FT} and \cite{KV}. In particular, we remark its key role in the proof of the Landau damping for the Vlasov-Poisson system given by Bedrossian, Masmoudi and Mouhot in \cite{BMM}. In fact, the problem studied in this paper is mentioned in \cite{BMM}. 

Motivated by the approach used in \cite{KV} and \cite{LO} to study the propagation of analytic regularity for the 2D Euler system, we prove the propagation of Gevrey regularity for the Vlasov-Poisson system on $\T^d$ as long as there exists a unique Sobolev solution $f$ for this system. Furthermore, we give a quantitative estimate for the growth in time of the Gevrey norm of the solution of the system in terms of the force field and the gradient in the velocity variables of the distribution of matter. As an application, we show global existence of Gevrey solutions for the Vlasov-Poisson system in $\T^3$. Furthermore, the propagation of Gevrey regularity can be easily modified to obtain the same result in $\R^d$. In particular, we obtain the global existence of analytic and Gevrey solutions for the Vlasov-Poisson system in $\R^3$.

\subsection{The main result}

Before presenting the main theorem of the paper, we proceed to define the key norms in its statement. \\

Let us start by introducing some basic notation. In the following, we denote the standard Sobolev norm of $f$ in $H^{\sigma}_{x,v}(\T^d\times \R^d)$ as $\|f\|_{\sigma}$ with $\sigma \in (0,\infty]$. We also consider some weighted Sobolev spaces in finite regularity denoted by $H^{\sigma}_{x,v;M}(\T^d\times \R^d)$ with norm 
\begin{equation}
    \|f\|_{\sigma,M}^2:=\sum_{|\alpha|\leq M} \|v^{\alpha}f\|_{ \sigma}^2.
\end{equation}  
Moreover, we define the Fourier coefficients of $f$ by \begin{equation}
    \hat{f}_k(\eta):=\dfrac{1}{(2\pi)^d}\int_{\T^d\times \R^d} e^{-ix\cdot k-iv\cdot \eta}f(x,v)dxdv,
\end{equation} 
and the Japanese bracket by $\langle k,\eta \rangle:=(1+|k|^2+|\eta|^2)^{1/2}$. \\

We proceed to define the function spaces in which we build solutions: the Gevrey classes.

\begin{definition}[Gevrey classes in $\T^d\times \R^d$]\label{gevrey}
A real-valued function $f\in C^{\infty}(\T^d\times \R^d)$ is said to be of \emph{Gevrey class $\frac{1}{s}$ with index of regularity $\lambda$, Sobolev correction $\sigma$ and weight $M$} if for some $s\in(0,1]$, $\sigma>0$ and $\lambda>0$, $f\in L^{2}(\T^d\times \R^d)$ and 
\begin{equation}
    \|f\|_{\lambda, \sigma,M;s}^2:=\sum_{|\alpha|\leq M} \|v^{\alpha}f\|_{\lambda, \sigma;s}^2<\infty,
\end{equation} 
where 
\begin{equation}
    \|v^{\alpha}f\|_{\lambda,\sigma;s}^2:=\|AD^{\alpha}_{\eta}\hat{f}_k(\eta)\|_{L^2_{k,\eta}}^2=\sum_{k\in\Z^d}\int_{\R_{\eta}^d}|D^{\alpha}_{\eta}\hat{f}_k(\eta)|^2\langle k,\eta \rangle^{2\sigma} e^{2\lambda \langle k,\eta \rangle^s} d\eta.
\end{equation} 
In the above $A$ is the Fourier multiplier $A=A_k(t,\eta)=\langle k,\eta \rangle^{\sigma}e^{\lambda(t)\langle k,\eta \rangle^{s}}.$ In the following, we denote by $\G^{\lambda,\sigma,M;1/s}(\T^d\times \R^d)$ to the set of functions of Gevrey class $s$ with index of regularity $\lambda$, Sobolev correction $\sigma$ and weight $M$.
\end{definition}

In order to propagate Gevrey regularity it is crucial to have a local existence result for this particular notion of regularity. Indeed, we use the following proposition which can be found in the beginning of Section 2.4 in \cite{BMM}.

\begin{proposition}[Local existence of Gevrey functions]\label{local_existence}
Let $f_0$ be an initial data for the Vlasov-Poisson system (\ref{vlasov-poisson}) on $\T^d\times \R^d$ (or $\R^d\times \R^d$) such that $\|f\|_{\lambda_0,\sigma,M;s}$ is finite for some $s\in (0,1]$, $\sigma >0$, $\lambda_0>0$ and $M>d/2$. Then, there exists some $T_0 > 0$ such that there exists a real continuous function $\lambda$ defined on $[0,T_0]$ with $\lambda(0)=\lambda_0$, $\inf_{t\in[0,T_0]} \lambda(t) > 0$, for which 
\begin{equation}
    \sup_{t\in [0,T_0]}\|f\|_{\lambda(t),\sigma,M;s}<\infty,
\end{equation} 
where $f$ is the unique smooth solution of the Vlasov-Poisson system (\ref{vlasov-poisson}) on $[0, T_0]$.
\end{proposition}

The main result in the paper is based on the following Beale-Kato-Majda type blow up criteria (see \cite{BKM}) which can be found in the beginning of Section 2.4 in \cite{BMM}.

\begin{proposition}[Blow up criterion]\label{blow_up_criterion}
Let $f_0$ be an initial data for the Vlasov-Poisson system (\ref{vlasov-poisson}) on $\T^d\times \R^d$ (or $\R^d\times \R^d$) such that $\|f\|_{\lambda_0,\sigma,M;s}$ is finite for some $s\in (0,1]$, $\sigma >0$, $\lambda_0>0$ and $M>d/2$. Let $T_{\max}>0$ be the maximal time of existence of the unique smooth solution $f$ of the Vlasov-Poisson system (\ref{vlasov-poisson}). Then, if $T_{\max}$ is finite 
\begin{equation}\label{blow_up}
    \limsup_{t\to T_{\max}} \Big(\|F[f]\|_{W^{1,\infty}}(t)+\|\nabla_v f\|_{\infty,M}(t)\Big)=\infty.
\end{equation}
\end{proposition}

Because of Proposition \ref{blow_up_criterion}, the propagation of Gevrey regularity on $[0,T_0]$ will follow if (\ref{blow_up}) is satisfied. We are now able to state the main result of the paper.

\begin{theorem}[Propagation of Gevrey regularity] \label{propagation_regularity}
Let $f_0$ be an initial data for the Vlasov-Poisson system (\ref{vlasov-poisson}) on $\T^d\times \R^d$ (or $\R^d\times \R^d$) such that $\|f_0\|_{\lambda_0,\sigma,M;s}$ is finite for some $s\in (0,1]$, $\lambda_0>0$, $\sigma \geq d/2+6$ and $M>d/2$. Then, the unique solution $f\in C(0,T_{\max};H^\sigma_{x,v;M})$ to the Vlasov-Poisson system (\ref{vlasov-poisson}) satisfies $\forall t\in[0,T_{\max})$, 
\begin{equation}
    \|f\|_{\lambda,\sigma,M;s}(t)\leq C\exp\Big\{Ct+Ct\exp\Big[C\int_0^t\Big(\|F[f]\|_{W^{1,\infty}}+\|\nabla_v f\|_{\infty,M}+1\Big) ds\Big] \Big\}=:A(t),
\end{equation}  
and $\forall t\in[0,T_{\max})$, \begin{equation}
    \lambda(t)\geq C\exp\Big[-\int_0^t\Big(2A(t)+1\Big)ds\Big],
\end{equation} 
where $T_{\max}$ is the maximal time of existence and $C$ is a constant depending on the initial data $f_0$, the dimension $d$, the Sobolev correction $\sigma$ and the weight $M$.
\end{theorem}

 This result is analogous to the one proved in \cite{KV} and \cite{LO} for the 2D Euler system. Even though Theorem \ref{propagation_regularity} applies to the Vlasov-Poisson system (\ref{vlasov-poisson}) in $\T^d \times \R^d$ and $\R^d \times \R^d$, we write down just the proof in the case of $\T^d\times \R^d$ since the other case follows identically. \\
 As an application of Theorem \ref{propagation_regularity}, we obtain the global existence of Gevrey solutions for the Vlasov-Poisson system (\ref{vlasov-poisson}) in $\T^3\times \R^3$ or $\R^3\times \R^3$. This result follows directly by using the global existence theorems in \cite{BR} and \cite{Pf} for the case of Vlasov-Poisson on $\T^d \times \R^d$ and $\R^d \times \R^d$ respectively. 


\begin{theorem}[Global existence of Gevrey solutions for $s\in (0,1)$]\label{global_existence_gevrey}
Let $f_0$ be a compactly supported initial data for the Vlasov-Poisson system (\ref{vlasov-poisson}) on $\T^3\times \R^3$ (or $\R^3\times \R^3$) such that $\|f_0\|_{\lambda_0,\sigma,M;s}$ is finite for some $s\in (0,1)$, $\sigma \geq 15/2$, $M>3/2$ and $\lambda_0>0$. Then, the unique global solution $f\in C(0,\infty;H^\sigma_{x,v;M})$ to the Vlasov-Poisson system (\ref{vlasov-poisson}) satisfies $\forall t\in[0,\infty)$, 
\begin{equation}
    \|f\|_{\lambda,\sigma,M;s}(t)\leq C\exp\Big\{Ct+Ct\exp\Big[C\int_0^t\Big(\|F[f]\|_{W^{1,\infty}}+\|\nabla_v f\|_{\infty,M}+1\Big) ds\Big] \Big\}=:A(t),
\end{equation} 
and $\forall t\in[0,\infty)$, 
\begin{equation}
    \lambda(t)\geq C\exp\Big[-\int_0^t\Big(2A(t)+1\Big)ds\Big],
\end{equation} 
where $C$ is a constant depending on the initial data $f_0$, the dimension $d$, the Sobolev correction $\sigma$ and the weight $M$. 
\end{theorem}

\begin{remark}
 Furthermore, since the global existence theorem proved by Lions-Perthame does not require the initial data to be compactly supported, we also obtain a global existence theorem for analytic functions for the Vlasov-Poisson system in $\R^3$. Therefore, this paper answers the open question coming from \cite{Be} on the existence of global analytic solutions for this system on $\R^3$.
\end{remark}




\begin{outline}
In Section \ref{preliminaries}, we introduce the relevant notions and tools used in the proof of Theorem \ref{propagation_regularity}. In Section \ref{growth_weighted_Sobolev}, we prove a quantitative estimate for the growth of the Sobolev norm of the distribution of matter $f$ which is used in the proof of Theorem \ref{propagation_regularity} and Theorem \ref{global_existence_gevrey}. Finally, we prove Theorem \ref{propagation_regularity} in Section \ref{section_main} and Theorem \ref{global_existence_gevrey} in Section \ref{section_application}. 
\end{outline}

\begin{ack}
I would like to thank my supervisors Mihalis Dafermos and Cl\'ement Mouhot for their continued guidance and encouragements throughout all this work. Also, I would like to thank Davide Parise for many suggestions during the revision of the manuscript. This article was supported by CONICYT Chile, the Cambridge Centre for Analysis and Cambridge Trust.
\end{ack}

\section{Preliminaries} \label{preliminaries}
Throughout this section, we discuss the tools required for the proof of our main results. Firstly, we state the Vlasov-Poisson system in Fourier variables and explain how we apply Proposition \ref{blow_up_criterion} about local existence. Secondly, we give more details on Gevrey spaces and state the main technical tools to prove the main result.

\subsection{Vlasov-Poisson in Fourier variables}
Let us recall the definition of the Fourier transform of a function $f$ defined on $\T^d\times \R^d$ by \begin{equation}
    \hat{f}_k(\eta):=\dfrac{1}{(2\pi)^d}\int_{\T^d\times \R^d} e^{-ix\cdot k-iv\cdot \eta}f(x,v)dxdv.
\end{equation} 
Similarly, we recall the definition of the Fourier transform of a function $\rho$ defined on $\T^d$ by \begin{equation}
    \hat{\rho}_k:=\dfrac{1}{(2\pi)^d}\int_{\T^d} e^{-ix\cdot k}\rho(x)dx.
\end{equation} 
Since we follow a Fourier space method, we use the Vlasov-Poisson in Fourier variables given by \begin{equation}
    \partial_t \hat{f}_k(t,\eta)-k\cdot \nabla_{\eta}\hat{f}_k(t,\eta)+\sum_{l\in \Z^d_*}\rho_l(t)\widehat{W}(l)l\cdot \eta \hat{f}_{k-l}(t,\eta)=0,
\end{equation} 
where $\rho_l$ are the Fourier coefficients of $\rho$ on $\T^d$. Moreover, recall that in Fourier variables for every $k\in \Z^d\setminus\{0\}$ we have the decay 
\begin{equation}
    |\widehat{W}(k)|\leq \dfrac{C_{W}}{|k|^2},
\end{equation} 
where $W$ is either the Coulomb or the Newtonian potential and $C_W$ is a constant depending just on $W$. We refer the reader to \cite{BMM} for further details about the periodisation in $\T^d$.

\subsection{Local existence}
As we discussed in the introduction, our main result relies on a local existence result and a blow up criterion which allow us to prove the theorem as long as 
\begin{equation} \label{condition_force}
    \limsup_{t\to T}\Big(\|F[f]\|_{W^{1,\infty}}(t)+\|\nabla_v f\|_{\infty,M}(t)\Big) <\infty.
\end{equation}
Proposition \ref{local_existence} and Proposition \ref{blow_up_criterion} are stated and discussed in Section 2.4 of \cite{BMM}. Even though there is no proof of this result in \cite{BMM}, Proposition \ref{local_existence} follows by standard methods, so we do not go into details. However, in \cite{BMM} the result is slightly different in the sense that the condition (\ref{condition_force}) is replaced by \begin{equation}
\limsup_{t\to T}\|f\|_{\sigma,M}(t) <\infty.    
\end{equation} 
Note that last condition is weaker since by the following lemma $\|f\|_{\sigma,M}(t)$ can be bounded uniformly in terms of $\|F[f]\|_{W^{1,\infty}}(t)+\|\nabla_v f\|_{\infty,M}(t)$. 

\begin{lemma}\label{sobolev_growth} Let $f_0$ be an initial data for the Vlasov-Poisson system (\ref{vlasov-poisson}) on $\T^d\times \R^d$ (or $\R^d\times \R^d$) such that $\|f_0\|_{\sigma,M}$ is finite for some $\sigma >0$ and $M>d/2$. Then, the unique solution $f\in C(0,T_0;H^\sigma_{x,v;M})$ to the Vlasov-Poisson system (\ref{vlasov-poisson}) satisfies 
\begin{equation}
\forall t\in[0,T_{\max}),\quad \|f\|_{\sigma,M}^2(t)\leq C \exp\Big[C\int_0^t\Big(\|F[f]\|_{W^{1,\infty}}+\|\nabla_v f\|_{\infty,M}+1\Big) ds\Big],    
\end{equation} 
where $T_{\max}$ is the maximal time of existence and $C$ is a constant depending on the initial data $f_0$, the dimension $d$, the Sobolev parameter $\sigma$ and the weight $M$.
\end{lemma}

We prove Lemma \ref{sobolev_growth} in Section \ref{growth_weighted_Sobolev}.
\subsection{Gevrey spaces}
Let us discuss the main properties of the Gevrey spaces. As mentioned in the introduction, the Gevrey regularity is a very strong notion of regularity satisfied by smooth functions which are close to be analytic. \\

The space $\G^{\lambda,\sigma,M;s}(\T^d\times \R^d)$ is a vector space closed under multiplication and differentiation. Moreover, $\G^{\lambda,\sigma,M;s_1}(\T^d\times \R^d)\subset \G^{\lambda,\sigma,M;s_2}(\T^d\times \R^d)$ whenever $s_1<s_2$. In particular, these properties are satisfied by the space of real analytic functions, which is exactly $\G^{\lambda,\sigma,M;1}(\T^d\times \R^d)$. As a matter of fact, in this case $\|f\|_{\lambda,\sigma,M; 1}$ is a norm for the space of analytic functions with radius of analyticity $\lambda$. \\

Similarly to the functions defined on $\T^d\times \R^d$, we introduce the corresponding Gevrey norms for functions on $\T^d$. 

\begin{definition}[Gevrey classes in $\T^d$]\label{spatial_gevrey}
A real-valued function $\rho\in C^{\infty}(\T^d)$ is said to be of \emph{Gevrey class $\frac{1}{s}$ with index of regularity $\lambda$ and Sobolev correction $\sigma$} if for some $s\in(0,1]$, $\sigma>0$ and $\lambda>0$, $\rho \in L^{2}(\T^d)$ and 
\begin{equation}
\|\rho\|_{\lambda,\sigma; s}^2:=\|B\rho_k\|_{L^2_{k}}=\sum_{k\in \Z^d}|\rho_k|^2 \langle k\rangle^{2\sigma} e^{2\lambda \langle k \rangle^{s}}< \infty,    
\end{equation} 
where $B$ is the Fourier multiplier $B=B_k(t)=\langle k \rangle^{\sigma}e^{\lambda(t)\langle k \rangle^s}$, the term $\langle k \rangle:=(1+|k|^2)^{1/2}$ is the Japanese bracket and $\rho_k$ the Fourier coefficients of $\rho$ on $\T^d$. 
\end{definition}

Henceforth, we omit the parameter $s\in(0,1]$ which is fixed from now on.

\subsection{Toolbox}
Before moving to the next section, we are going to state two important lemmata which are going to be used repetitively in the main energy estimates of the article. \\
In the following, we use the notation $f\lesssim g$ when there exists a constant $C > 0$ independent of the parameters of interest such that $f \leq Cg$. Moreover, we use the notation $f \lesssim_{\alpha} g$ if we want to emphasize that the implicit constant depends on $\alpha$. 

\begin{lemma}[Young's inequality]\label{young}
We have the following two inequalities. 
\begin{enumerate}
    \item Let $f^1_k(\eta)$, $f^2_k (\eta) \in L^2(\Z^d \times \R^d )$ and $\langle k\rangle^{\sigma} r_k (t) \in L^2(\Z^d)$ for $\sigma >d/2$. Then, for any $t \in \R$ we have 
    \begin{equation}
        \Big|\sum_{k,l}\int_{\R_{\eta}^d}f^1_k(\eta)r_l(t) f^2_{k-l}(\eta)d\eta\Big|\lesssim_{\sigma,d} \|f^1\|_{L^2_{k,\eta}} \| \langle k\rangle^{\sigma} r\|_{L^2_{k}} \|f^2\|_{L^2_{k,\eta}}.
    \end{equation}
    \item Let $f^1_k(\eta)$, $\langle k\rangle^{\sigma} f^2_k (\eta) \in L^2(\Z^d \times \R^d )$ and $ r_k (t) \in L^2(\Z^d)$ for $\sigma >d/2$. Then, for any $t \in \R$ we have 
    \begin{equation}
        \Big|\sum_{k,l}\int_{\R_{\eta}^d}f^1_k(\eta)r_l(t) f^2_{k-l}(\eta)d\eta\Big|\lesssim_{\sigma,d} \|f^1\|_{L^2_{k,\eta}} \| r\|_{L^2_{k}} \|\langle k\rangle^{\sigma}f^2\|_{L^2_{k,\eta}}.
    \end{equation}
\end{enumerate}
\end{lemma}



The proof of Lemma \ref{young} can be found in Section 3.1 of \cite{BMM}. Finally, we state a result which is going to be used to compare Gevrey norms of the spatial density and Gevrey norms for the distribution of matter.

\begin{lemma}\label{spatial_density}
Let $f$ be a solution of the Vlasov-Poisson system (\ref{vlasov-poisson}) and $\rho$ the induced spatial density. Then, 
\begin{equation}
    \|\rho\|_{\lambda, \sigma;s}\lesssim \|f\|_{\lambda, \sigma, M;s},
\end{equation} 
for every $M>d/2$
\end{lemma}
\begin{proof}
Since $\hat{\rho}_k=\hat{f}_k(0)$, we have
\begin{align}
    \|\rho\|_{\lambda, \sigma;s}&=\sum_{k\in \Z^d}|\hat{f}_k(0)|^2 \langle k\rangle^{2\sigma} e^{2\lambda \langle k \rangle^{s}},\\
    &\lesssim \sum_{|\alpha|\leq M}\sum_{k\in \Z^d}\langle k\rangle^{2\sigma} e^{2\lambda \langle k \rangle^{s}} \int_{\R^d_{\eta}} |D^{\alpha}_{\eta}\hat{f}_k(\eta)|d\eta,\\&
    \lesssim \sum_{|\alpha|\leq M}\sum_{k\in \Z^d}\int_{\R^d_{\eta}}|D^{\alpha}_{\eta}\hat{f}_k(\eta)|\langle k,\eta \rangle^{2\sigma} e^{2\lambda \langle k,\eta \rangle^s} d\eta,\\
    &\lesssim \|f\|_{\lambda, \sigma, M;s},
\end{align}
where the Sobolev embedding $H^{d/2+}(\R^d)\hookrightarrow C^0(\R^d)$ was used in the first inequality.
\end{proof}

\section{Growth of weighted Sobolev norms} \label{growth_weighted_Sobolev}
Let us prove the quantitative bound for the growth of the Sobolev norm of the solution $f$ of the Vlasov-Poisson system which will be used in the proof of Theorem \ref{propagation_regularity}.


\begin{proof}[Proof of Lemma \ref{sobolev_growth}]
By Gronwall's inequality and adding over the indexes $\alpha$ and $\beta$, it is enough to show that 
\begin{equation}
    \dfrac{1}{2}\dfrac{d}{dt}\|D^{\beta}(v^{\alpha}f)\|_{L^2}^2\leq C \sum_{|\alpha|\leq M}\sum_{|\beta|\leq \sigma}\|D^{\beta}(v^{\alpha}f)\|_{L^2}^2\Big(\|F[f]\|_{W^{1,\infty}}+\|\nabla_v f\|_{\infty,M}+1\Big).
\end{equation} 
By the Vlasov-Poisson system, 
\begin{align}
    \dfrac{1}{2}\dfrac{d}{dt}\|D^{\beta}(v^{\alpha}f)\|_{L^2}^2=&\int_{\T^d\times \R^d} D^{\beta}(v^{\alpha}f) D^{\beta}(v^{\alpha}\partial_t f) dxdv,\\
    =&-\int_{\T^d\times \R^d} D^{\beta}(v^{\alpha}f) D^{\beta}(v^{\alpha}(v\cdot \nabla_x  f)) dxdv\\
    &-\int_{\T^d\times \R^d} D^{\beta}(v^{\alpha}f) D^{\beta}(v^{\alpha}(F \cdot \nabla_v f)) dxdv,\nonumber\\
    =:&-E_L-E_{NL},
\end{align} where the first term $E_L$ has the linear contribution $v\cdot \nabla_x f$ coming from the free transport and $E_{NL}$ has the nonlinear contribution $F\cdot \nabla_v f$. Firstly, we estimate the linear term $E_L$. 

\begin{claim}\label{linear_term_sobolev}
\begin{equation}
    |E_{L}|\lesssim \|f\|_{\sigma,M}^2.
\end{equation}
\end{claim}

\begin{proof}
Since,
\begin{align}
    E_{L}=& \int_{\T^d\times \R^d} D^{\beta}(v^{\alpha}f) D^{\beta}(v^{\alpha}(v\cdot \nabla_x  f)) dxdv,\\
    =&\sum_{i=1}^d\sum_{\beta_1+\beta_2=\beta}\binom{\beta}{\beta_1}\int_{\T^d\times \R^d} D^{\beta}(v^{\alpha}f) D^{\beta_1}(v_i)\partial_{x_i}D^{\beta_1}(v^{\alpha} f) dxdv,
\end{align}
there are two options: $D^{\beta_1}(v_i)=1$ or $D^{\beta_1}(v_i)=v_i$. In the former case 
\begin{align}
    \Big| \int_{\T^d\times \R^d} &D^{\beta}(v^{\alpha}f) D^{\beta_1}(v_i)\partial_{x_i}D^{\beta_1}(v^{\alpha} f) dxdv \Big|\\
    \leq& \int_{\T^d\times \R^d} |D^{\beta}(v^{\alpha}f)|^2dxdv +\int_{\T^d\times \R^d}|\partial_{x_i}D^{\beta_1}(v^{\alpha} f)|^2 dxdv,\nonumber\\
    \leq& \sum_{|\alpha|\leq M}\sum_{|\beta|= \sigma}\|D^{\beta}(v^{\alpha}f)\|_{L^2}^2,
\end{align}
since in this case $|\beta_1|\leq \sigma-1$. Otherwise, if $D^{\beta_1}(v_i)=v_i$ then
\begin{align}
     \int_{\T^d\times \R^d} D^{\beta}(v^{\alpha}f) D^{\beta_1}(v_i)\partial_{x_i}D^{\beta_1}(v^{\alpha} f) dxdv=&
     \int_{\T^d\times \R^d} v_i D^{\beta}(v^{\alpha}f) \partial_{x_i}D^{\beta}(v^{\alpha} f) dxdv,\\
     =&\dfrac{1}{2} \int_{\T^d\times \R^d}  v_i \partial_{x_i}\Big(D^{\beta}(v^{\alpha} f)^2\Big) dxdv, \\
     =&0.
\end{align}
    Hence, 
    \begin{equation}
        |E_{L}|\lesssim  \sum_{|\alpha|\leq M}\sum_{|\beta|\leq \sigma}\|D^{\beta}(v^{\alpha}f)\|_{L^2}^2\lesssim  \|f\|_{\sigma,M}^2.
    \end{equation}
\end{proof}

Before moving on to the estimate of the nonlinear term $E_{NL}$, we prove a basic inequality which will be used afterwards. 

\begin{claim} \label{interpolation}
For every $u\in W^{1,\infty}(\T^d\times \R^d)\cap H^{\sigma}(\T^d\times \R^d)$ and $v\in L^{\infty}(\T^d\times \R^d)\cap H^{\sigma}(\T^d\times \R^d)$,
\begin{equation}
    \sum_{|\alpha|\leq \sigma }\|D^{\alpha}(uv)-uD^{\alpha}v\|_{L^2}\leq C\Big(\|\nabla u\|_{\infty}\|v\|_{\sigma-1}+\|u\|_{\sigma}\| v\|_{\infty}\Big), 
\end{equation}
where $C$ is a uniform constant depending on the Sobolev parameter $\sigma$ and the dimension $d$.
\end{claim}

\begin{proof}
Since for every $u$, $v\in H^{\sigma}(\T^d\times \R^d)$ we have
\begin{equation}
    \|uv\|_{\sigma} \leq \|u\|_{\infty} \|v\|_{\sigma}+\|u\|_{\sigma} \|v\|_{\infty},
\end{equation}
then 
\begin{align}
    \sum_{|\alpha|\leq \sigma }\|D^{\alpha}(uv)-uD^{\alpha}v\|_{L^2}&\leq c_{\alpha}\sum_{\beta}\|D^{\beta}uD^{\alpha-\beta}v\|_{L^2},\\ 
    &\leq C\Big(\|\nabla u\|_{\infty}\|v\|_{\sigma-1}+\|u\|_{\sigma}\| v\|_{\infty}\Big).
\end{align}
\end{proof}

Secondly, we estimate the nonlinear term $E_{NL}$. 

\begin{claim}\label{nonlinear_term_sobolev}
\begin{equation}
    |E_{NL}|\lesssim \|f\|_{\sigma,M}^2\Big(\|F[f]\|_{W^{1,\infty}}+\|\nabla_v f\|_{\infty,M}\Big).
\end{equation}
\end{claim}
\begin{proof} 
Let us decompose $E_{NL}$ as
\begin{align}
    E_{NL}=& \int_{\T^d\times \R^d} D^{\beta}(v^{\alpha}f) D^{\beta}(v^{\alpha}(F\cdot \nabla_v  f)) dxdv,\\
    =&\sum_{i=1}^d\int_{\T^d\times \R^d} D^{\beta}(v^{\alpha}f) \Big(D^{\beta}(v^{\alpha}F_i \partial_{v_i}  f)-F_i D^{\beta}(v^{\alpha}\partial_{v_i}  f)\Big)dxdv\\
    &+\sum_{i=1}^d\int_{\T^d\times \R^d}  F_i D^{\beta}(v^{\alpha}f) D^{\beta}(v^{\alpha}\partial_{v_i}  f)dxdv,\nonumber\\
    =:&E_{NL}^1+E_{NL}^2.
\end{align}
Firstly, we estimate $E_{NL}^1$ by applying (\ref{interpolation}),
\begin{align}
    |E_{NL}^1|&\lesssim \|D^{\beta}(v^{\alpha}f)\|_{L^2}\Big(\|\nabla F\|_{\infty}\|v^{\alpha}\partial_{v_i}f\|_{\sigma-1}+\|F\|_{\sigma}\| v^{\alpha}\partial_{v_i}f\|_{\infty}\Big),\\
    &\lesssim \|f\|_{\sigma,M}\Big(\|\nabla F\|_{\infty}\|f\|_{\sigma,M}+\|\rho\|_{\sigma-1}\|\nabla_v f\|_{\infty,M} \Big),\\
    &\lesssim \|f\|_{\sigma,M}^2\Big(\|\nabla F\|_{\infty}+\|\nabla_v f\|_{\infty,M} \Big),
\end{align}
where we have used that $\|F\|_{\sigma}\lesssim \|\rho\|_{\sigma-1}$ by elliptic regularity and $\|\rho\|_{\sigma-1}\lesssim \|f\|_{\sigma-1,M}$ by taking $\lambda=0$ in Lemma \ref{spatial_density}. Secondly, we estimate $E_{NL}^2$ by using a commutator estimate,
\begin{align}
    |E_{NL}^2|=&\Big| \sum_{i=1}^d\int_{\T^d\times \R^d}  F_i D^{\beta}(v^{\alpha}f) D^{\beta}(v^{\alpha}\partial_{v_i}  f)dxdv \Big|, \\
    =&\Big| \sum_{i=1}^d\int_{\T^d\times \R^d}  F_i D^{\beta}(v^{\alpha}f) \Big[\partial_{v_i}\Big(D^{\beta}(v^{\alpha} f)\Big)-D^{\beta}(v^{\alpha-i} f)\Big]dxdv \Big|, \\
    =& \dfrac{1}{2}\Big| \sum_{i=1}^d\int_{\T^d\times \R^d}  F_i \partial_{v_i}\Big(D^{\beta}(v^{\alpha} f)^2\Big)dxdv \Big|\\
    &+\Big| \sum_{i=1}^d\int_{\T^d\times \R^d}  F_i D^{\beta}(v^{\alpha}f) D^{\beta}(v^{\alpha-i} f)dxdv \Big|,\nonumber\\
    \lesssim&\sum_{i=1}^d\|F\|_{\infty}\int_{\T^d\times \R^d} \Big( |D^{\beta}(v^{\alpha}f)|^2+ |D^{\beta}(v^{\alpha-i} f)|^2\Big)dxdv,\\
    \lesssim& \|F\|_{\infty} \|f\|_{\sigma,M}^2.
\end{align}
Therefore, 
\begin{equation}
    |E_{NL}|\leq |E_{NL}^1|+|E_{NL}^2|\leq \|f\|_{\sigma,M}^2\Big(\|F[f]\|_{W^{1,\infty}}+\|\nabla_v f\|_{\infty,M}\Big).
\end{equation}

\end{proof}

Consequently, by Claim \ref{linear_term_sobolev} and Claim \ref{nonlinear_term_sobolev}, \begin{equation}
    \dfrac{1}{2}\dfrac{d}{dt}\|D^{\beta}(v^{\alpha}f)\|_{L^2}^2\leq C \|D^{\beta}(v^{\alpha}f)\|_{L^2}^2\Big(\|F[f]\|_{W^{1,\infty}}+\|\nabla_v f\|_{\infty,M}+1\Big).
\end{equation}
\end{proof}


\section{Proof of Theorem \ref{propagation_regularity}: Propagation of Gevrey regularity} \label{section_main}
In the following, we consider an index of regularity $\lambda(t)$ and we show that it can be taken to be a decreasing positive function for which $f(t)\in \G^{\lambda(t),\sigma,M}$ for every $t>0$ and such that its decay is controlled as in Theorem \ref{propagation_regularity}. Indeed, we will show that there exists a continuous positive function $A(t)$ such that $\|f\|_{\lambda, \sigma,M}\leq A(t)$ for every $t$ and \begin{equation}
\lambda(t)= C\exp\Big[-\int_0^t\Big(2A(t)+1\Big)ds\Big]    
\end{equation} 
is such that $f(t)\in \G^{\lambda(t),\sigma,M}$ for every $t>0$.

The proof of Theorem \ref{propagation_regularity} follows by energy estimates. More precisely, we bound \begin{equation}
\dfrac{1}{2}\dfrac{d}{dt}\|f\|_{\lambda, \sigma,M}^2=\dfrac{1}{2}\sum_{|\alpha|\leq M}\dfrac{d}{dt}\|Av^{\alpha}f\|^2_2,    
\end{equation} 
in terms of $\|f\|_{\lambda, \sigma,M}^2$ and $\|f\|_{ \sigma,M}$ in order to conclude using Gronwall's inequality. The main technical difficulty consists in decomposing the nonlinear term $E_{NL}$ which appears after using the Vlasov-Poisson system in $\frac{1}{2}\frac{d}{dt}\|f\|_{\lambda, \sigma,M}^2$. This problem is overcome using elements of the proof of the analogue of Theorem \ref{propagation_regularity} for the 2D Euler system. Indeed, we apply a commutator estimate coming from \cite{LO} and  similar estimates to the ones used in \cite{KV} to improve the estimates in \cite{LO}. 

In a first stage, we decompose the previous derivative in three pieces, a linear contribution $E_L$ coming from the free transport term $v\cdot \nabla_x f$, a nonlinear contribution $E_{NL}$ coming from the nonlinearity $F(t,x) \cdot \nabla_v f$ and a higher order term $CK$ (for Cauchy-Kovalevskaya) which is used to absorb the highest order terms in the estimates. More precisely, 
\begin{align}
\dfrac{1}{2}\sum_{|\alpha|\leq M}\dfrac{d}{dt}\|Av^{\alpha}f\|^2_2 = & \dfrac{1}{2}\sum_{|\alpha|\leq M}\dfrac{d}{dt}\|AD^{\alpha}_{\eta}\hat{f}\|^2_2,\\
= & \dot{\lambda}\sum_{|\alpha|\leq M}\sum_{k\in \Z^d}\int_{\R_{\eta}^d}\langle k,\eta \rangle^s|AD^{\alpha}_{\eta}\hat{f}_k(t,\eta)|^2 d\eta\\
& +\sum_{|\alpha|\leq M}\sum_{k\in \Z^d}\int_{\R_{\eta}^d}AD^{\alpha}_{\eta}\overline{\hat{f}_k(t,\eta)} AD^{\alpha}_{\eta}\partial_t \hat{f}_k(t,\eta) d\eta, \nonumber\\
= & \dot{\lambda}\sum_{|\alpha|\leq M}\sum_{k\in \Z^d}\int_{\R_{\eta}^d} |D^{\alpha}_{\eta}\hat{f}_k(t,\eta)|^2 \langle k,\eta \rangle^{2\sigma+s}e^{2\lambda\langle k,\eta \rangle^s} d\eta\\ 
&+\sum_{|\alpha|\leq M}\sum_{k\in \Z^d}\int_{\R_{\eta}^d}AD^{\alpha}_{\eta}\overline{\hat{f}_k(t,\eta)} AD^{\alpha}_{\eta}[k\cdot \nabla_{\eta}\hat{f}_k(t,\eta)] d\eta\nonumber\\
&-\sum_{|\alpha|\leq M}\sum_{k\in \Z^d}\int_{\R_{\eta}^d}AD^{\alpha}_{\eta}\overline{\hat{f}_k(t,\eta)}\nonumber\\
&\qquad \qquad \qquad\cdot AD^{\alpha}_{\eta}\Big[\sum_{l\in \Z^d_*}\rho_l(t)\widehat{W}(l)l\cdot \eta \hat{f}_{k-l}(t,\eta)\Big] d\eta,\nonumber\\
=: & CK-E_L-E_{NL},
\end{align}
where $CK:=\dot{\lambda}\|f\|_{\lambda,\sigma+s/2,M}^2(t)$. The core of the proof rests on the three claims proved in this section which estimate both terms $E_L$ and $E_{NL}$. At the end of this section we combine all the results to give a proof of Theorem \ref{propagation_regularity}. 

\subsection{Estimate of the linear term $E_{L}$} 

\begin{claim}\label{linear_term_gevrey}
\begin{equation}
    |E_L|\lesssim \lambda\|v^{\alpha}f\|_{\lambda,\sigma+1/2}^2+\|v^{\alpha}f\|_{\lambda,\sigma}^2.
\end{equation}
\end{claim}

\begin{proof}
We estimate $E_L$ using a commutator estimate as follows,
\begin{align}
    E_L&=-\sum_{k\in \Z^d}\int_{\R_{\eta}^d}AD^{\alpha}_{\eta}\overline{\hat{f}_k(t,\eta)} AD^{\alpha}_{\eta}[k\cdot \nabla_{\eta}\hat{f}_k(t,\eta)] d\eta,\\
    &=-\sum_{i=1}^d\sum_{k\in \Z^d}\int_{\R_{\eta}^d} \partial_{\eta_i}\Big(\dfrac{|D^{\alpha}_{\eta}\hat{f}_k(t,\eta)|^2}{2}\Big)A^2k_i d\eta,\\
    &=\sum_{i=1}^d\sum_{k\in \Z^d}\int_{\R_{\eta}^d} |D^{\alpha}_{\eta}\hat{f}_k(t,\eta)|^2A\partial_{\eta^i}Ak_i d\eta, \\ &=\sum_{i=1}^d\sum_{k\in \Z^d}\int_{\R_{\eta}^d} |D^{\alpha}_{\eta}\hat{f}_k(t,\eta)|^2A(\sigma\eta_i\langle k,\eta \rangle^{\sigma-2}e^{\lambda\langle k,\eta \rangle^s}\\
    &\qquad +\lambda s\eta_i\langle k,\eta \rangle^{\sigma+s-2}e^{\lambda\langle k,\eta \rangle^s})k_i d\eta.\nonumber
\end{align}

Consequently, 
\begin{align}
    |E_L|\lesssim & \sum_{i=1}^d\sum_{k\in \Z^d}\int_{\R_{\eta}^d} |D^{\alpha}_{\eta}\hat{f}_k(t,\eta)|^2 \langle k,\eta \rangle^{2\sigma}e^{2\lambda\langle k,\eta \rangle^s} d\eta\\
    &+\lambda\sum_{i=1}^d\sum_{k\in \Z^d}\int_{\R_{\eta}^d} |D^{\alpha}_{\eta}\hat{f}_k(t,\eta)|^2 \langle k,\eta \rangle^{2\sigma+s}e^{2\lambda\langle k,\eta \rangle^{s}} d\eta,\nonumber\\
    \lesssim & \|v^{\alpha}f\|^2_{\lambda,\sigma}+\lambda\|v^{\alpha}f\|^2_{\lambda,\sigma+s/2}.
\end{align}
\end{proof}
\subsection{Estimate of the nonlinear term $E_{NL}$} We decompose $E_{NL}$ after applying $D^{\alpha}_{\eta}$ to the nonlinearity term of the Vlasov-Poisson system. Hence,

\begin{align}
E_{NL}=&\sum_{k\in \Z^d}\int_{\R_{\eta}^d}AD^{\alpha}_{\eta}\overline{\hat{f}_k(t,\eta)} A \Big[\sum_{l\in \Z^d_*}\rho_l(t)\widehat{W}(l)l\cdot \eta D^{\alpha}_{\eta}\hat{f}_{k-l}(t,\eta)\Big] d\eta\\
&+\sum_{k\in \Z^d}\int_{\R_{\eta}^d}AD^{\alpha}_{\eta}\overline{\hat{f}_k(t,\eta)} A\Big[\sum_{|j|=1, j\leq \alpha}\sum_{l\in \Z^d_*}\rho_l(t)\widehat{W}(l)l_j D^{\alpha-j}_{\eta}\hat{f}_{k-l}(t,\eta)\Big] d\eta\nonumber,\\
=&:E_{NL}^1+E_{NL}^2,
\end{align}
where the first term has one derivative in $\eta$ more than the second one. Now, we proceed to estimate the non-linear contribution $E_{NL}^1$ where all the technical difficulties are concentrated. Actually, this is the point where we use the aforementioned commutator estimate. Indeed, since $E_{NL}^2$ is a lower order term, the estimates are less involved.

\begin{claim}\label{first_nonlinear_term_gevrey}
\begin{equation}
    |E_{NL}^1|\lesssim\|f\|_{\lambda, \sigma+s/2,M}^2\Big(\lambda^2\|f\|_{\lambda,\sigma,M}+\lambda\|f\|_{\sigma,M}\Big)+\|f\|_{\lambda,\sigma,M}^2\|f\|_{\sigma,M}\Big(\lambda^2+1\Big).
\end{equation}
\end{claim}

\begin{proof}
To deal with $E^1_{NL}$, we make use of the fact that the force $F$ does not depend on $x$, so that the nonlinearity has a divergence structure. In other words, we use the identity 
\begin{equation}
\int_{\R^d_{x}\times \R^d_{v}}F(t,x)\cdot \Big[\nabla_v\cdot \Big(Av^{\alpha}f\Big)^2\Big] dvdx=0.    
\end{equation} 
In particular,
\begin{align}
E_{NL}^1&=\sum_{k\in \Z^d}\int_{\R_{\eta}^d}AD^{\alpha}_{\eta}\overline{\hat{f}_k(t,\eta)} A \Big[\sum_{l\in \Z^d_*}\rho_l(t)\widehat{W}(l)l\cdot \eta D^{\alpha}_{\eta}\hat{f}_{k-l}(t,\eta)\Big] d\eta,\\
&=\sum_{k\in \Z^d}\sum_{l\in \Z^d_*}\int_{\R_{\eta}^d}AD^{\alpha}_{\eta}\overline{\hat{f}_k(t,\eta)} (A_k(t,\eta)-A_{k-l}(t,\eta)) \rho_l(t)\widehat{W}(l)l\cdot \eta D^{\alpha}_{\eta}\hat{f}_{k-l}(t,\eta) d\eta.
\end{align}

We decompose $E_{NL}^1$ in the two terms $E_{NL}^{1,1}$ and $E_{NL}^{1,2}$ which are given after using the identity, 
\begin{align}
    A_{k}(t,\eta)-A_{k-l}(t,\eta)=&\langle k,\eta \rangle^\sigma e^{\lambda\langle k,\eta \rangle^s}-\langle k-l,\eta \rangle^\sigma e^{\lambda\langle k-l,\eta \rangle^s},\\ =&(\langle k,\eta \rangle^\sigma-\langle k-l,\eta \rangle^\sigma)e^{\lambda\langle k-l,\eta \rangle^s}\\
    &+(e^{\lambda\langle k,\eta \rangle^s}-e^{\lambda\langle k-l,\eta \rangle^s})\langle k,\eta \rangle^\sigma,\nonumber
\end{align}   
in the same order as stated. By the mean value theorem,
\begin{align}
    \langle k,\eta \rangle^\sigma-\langle k-l,\eta \rangle^\sigma=&\sigma(\langle k,\eta \rangle-\langle k-l,\eta \rangle)\\
    & \cdot(\langle k-l,\eta \rangle+\theta_{k,l,\eta} (\langle k,\eta \rangle-\langle k-l,\eta \rangle))^{\sigma-1},\nonumber\\
    \leq& \sigma\langle l \rangle(\langle k-l,\eta \rangle+\langle l \rangle)^{\sigma-1},\\\leq& \sigma\langle l \rangle^\sigma+C\sigma\langle l \rangle \langle k-l,\eta \rangle(\langle l \rangle^{\sigma-2}+\langle k-l,\eta \rangle^{\sigma-2}),\\\leq & \sigma\langle l \rangle^\sigma+C\sigma \langle l \rangle^{\sigma-1}\langle k-l,\eta \rangle+C\sigma\langle l \rangle\langle k-l,\eta \rangle^{\sigma-1},
\end{align}
which is used several times in this section. We estimate $E_{NL}^{1,1}$ using the previous inequality,
\begin{align}
    |E_{NL}^{1,1}|\lesssim & \sum_{k,l}\int_{\R_{\eta}^d}\langle k,\eta \rangle^{\sigma}e^{\lambda \langle k,\eta \rangle^s}|D^{\alpha}_{\eta}\hat{f}_k(t,\eta)|\langle l \rangle^{\sigma-1} |\rho_l|\langle k-l,\eta \rangle e^{\lambda\langle k-l,\eta \rangle^s}|\eta| \\
    &\qquad \qquad \qquad \cdot  |D^{\alpha}_{\eta}\hat{f}_{k-l}(t,\eta)|d\eta \nonumber\\
    &+\sum_{k,l}\int_{\R_{\eta}^d}\langle k,\eta \rangle^{\sigma}e^{\lambda \langle k,\eta \rangle^s}|D^{\alpha}_{\eta}\hat{f}_k(t,\eta)|\langle l \rangle^{\sigma-2} |\rho_l|\langle k-l,\eta \rangle^2 e^{\lambda\langle k-l,\eta \rangle^s}\nonumber\\
    &\qquad \qquad \qquad \cdot |D^{\alpha}_{\eta}\hat{f}_{k-l}(t,\eta)|d\eta\nonumber\\
    &+\sum_{k,l}\int_{\R_{\eta}^d}\langle k,\eta \rangle^{\sigma}e^{\lambda \langle k,\eta \rangle^s}|D^{\alpha}_{\eta}\hat{f}_k(t,\eta)| |\rho_l|\langle k-l,\eta \rangle^{\sigma}e^{\lambda\langle k-l,\eta \rangle^s} |D^{\alpha}_{\eta}\hat{f}_{k-l}(t,\eta)|d\eta.\nonumber
\end{align}

Moreover, since $e^x\leq e+x^2e^x$ for every $x\geq 0$, then we can apply $e^{\lambda\langle k-l,\eta \rangle^s}\leq e+\lambda^2\langle k-l,\eta \rangle^{2s}e^{\lambda \langle k-l,\eta \rangle^s}$ in the first two terms. Hence, by Young's inequality
\begin{align}
|E_{NL}^{1,1}| \lesssim & \|v^{\alpha}f\|_{\lambda, \sigma} \|\rho\|_{\sigma-1} \|v^{\alpha}f\|_{d/2+2}+\lambda^2\|v^{\alpha}f\|_{\lambda, \sigma} \|\rho\|_{\sigma-1} \|v^{\alpha}f\|_{\lambda,d/2+5}\\ &+\|v^{\alpha}f\|_{\lambda, \sigma} \|\rho\|_{\sigma-2} \|v^{\alpha}f\|_{d/2+3}+\lambda^2\|v^{\alpha}f\|_{\lambda, \sigma} \|\rho\|_{\sigma-2} \|v^{\alpha}f\|_{\lambda, d/2+6}\nonumber\\
&+\|v^{\alpha}f\|_{\lambda, \sigma} \|\rho\|_{d/2+1} \|v^{\alpha}f\|_{\lambda, \sigma}.\nonumber
\end{align}
Furthermore, since $\sigma\geq d/2+6$
\begin{align}
|E_{NL}^{1,1}| &\lesssim \|v^{\alpha}f\|_{\lambda, \sigma}\|\rho\|_{ \sigma}\|v^{\alpha}f\|_{\sigma}+\lambda^2\|v^{\alpha}f\|_{\lambda, \sigma}^2\|\rho\|_{ \sigma}+\|v^{\alpha}f\|_{\lambda, \sigma}^2\|\rho\|_{\sigma}.
\end{align}

Now, we proceed to bound the term $E_{NL}^{1,2}$. We decompose the term $E_{NL}^{1,2}$ using 
\begin{align}
e^{\lambda\langle k,\eta \rangle^s}-e^{\lambda\langle k-l,\eta \rangle^s}=&(e^{\lambda(\langle k,\eta \rangle^s-\langle k-l,\eta \rangle^s)}-1-\lambda(\langle k,\eta \rangle^s-\langle k-l,\eta \rangle^s))e^{\lambda \langle k-l,\eta \rangle^s}\\
&+\lambda(\langle k,\eta \rangle^s-\langle k-l,\eta \rangle^s)\langle k-l,\eta \rangle^{\sigma-s/2}e^{\lambda \langle k-l,\eta \rangle^s}\langle k,\eta \rangle^{-(\sigma-s/2)}\nonumber\\
&+\lambda(\langle k,\eta \rangle^s-\langle k-l,\eta \rangle^s)(\langle k,\eta \rangle^{\sigma-s/2}-\langle k-l,\eta \rangle^{\sigma-s/2})\nonumber\\
    &\qquad \qquad \cdot e^{\lambda \langle k-l,\eta \rangle^s}\langle k,\eta \rangle^{-(\sigma-s/2)}.  \nonumber
\end{align}
Hence, we call $A$, $B$ and $C$ to the three terms in the previous identity which decompose $E_{NL}^{1,2}$. To bound $|A|$, we use the inequality $|e^x-1-x|\leq x^2e^{|x|}$ so that
\begin{align}
    |A|\lesssim & \lambda^2 \sum_{k,l}\int_{\R_{\eta}^d}\langle k,\eta \rangle^{2\sigma}e^{\lambda \langle k,\eta \rangle^s}|D^{\alpha}_{\eta}\hat{f}_k(t,\eta)|\langle l \rangle e^{\lambda \langle l \rangle^s}|\rho_l|e^{\lambda\langle k-l,\eta \rangle^s}\\
    &\qquad \qquad \qquad \cdot|\langle k,\eta \rangle^s-\langle k-l,\eta \rangle^s||\eta| |D^{\alpha}_{\eta}\hat{f}_{k-l}(t,\eta)|d\eta.\nonumber
\end{align}    
Now, we apply the mean value theorem, the inequality 
\begin{equation}
|\langle k,\eta \rangle^s-\langle k-l,\eta \rangle^s|\leq \dfrac{\langle l \rangle}{\langle k,\eta \rangle^{1-s}+\langle k-l,\eta \rangle^{1-s}},    
\end{equation} 
in the first term, and the triangle inequality $\langle k,\eta \rangle^{s} -\langle k-l,\eta \rangle^{s}\leq \langle s \rangle^{s}$ to the rest. Therefore, 
    \begin{align}
    |A|\lesssim &\lambda^2 \sum_{k,l}\int_{\R_{\eta}^d}\langle k,\eta \rangle^{\sigma+s/2}e^{\lambda \langle k,\eta \rangle^s}|D^{\alpha}_{\eta}\hat{f}_k(t,\eta)|\langle l \rangle e^{\lambda \langle l \rangle^s}|\rho_l|\langle k-l,\eta \rangle^{\sigma+3s/2-1}\\
    &\qquad \qquad \qquad \cdot e^{\lambda\langle k-l,\eta \rangle^s} |D^{\alpha}_{\eta}\hat{f}_{k-l}(t,\eta)|d\eta\nonumber\\
    &+\lambda^2 \sum_{k,l}\int_{\R_{\eta}^d}\langle k,\eta \rangle^{\sigma+s/2}e^{\lambda \langle k,\eta \rangle^s}|D^{\alpha}_{\eta}\hat{f}_k(t,\eta)|\langle l \rangle^{\sigma+3s/2-1} e^{\lambda \langle l \rangle^s}|\rho_l|\langle k-l,\eta \rangle\nonumber \\
    &\qquad \qquad \qquad \cdot e^{\lambda\langle k-l,\eta \rangle^s}|D^{\alpha}_{\eta}\hat{f}_{k-l}(t,\eta)|d\eta\nonumber
\end{align}
\begin{align}    
    &+\lambda^2 \sum_{k,l}\int_{\R_{\eta}^d}\langle k,\eta \rangle^{\sigma+s/2}e^{\lambda \langle k,\eta \rangle^s}|D^{\alpha}_{\eta}\hat{f}_k(t,\eta)|\langle l \rangle^{\sigma+3s/2-2} e^{\lambda \langle l \rangle^s}|\rho_l|\langle k-l,\eta \rangle^2\nonumber\\
    &\qquad \qquad \qquad \cdot e^{\lambda\langle k-l,\eta \rangle^s} |D^{\alpha}_{\eta}\hat{f}_{k-l}(t,\eta)|d\eta\nonumber\\
    &+\lambda^2 \sum_{k,l}\int_{\R_{\eta}^d}\langle k,\eta \rangle^{\sigma+s/2}e^{\lambda \langle k,\eta \rangle^s}|D^{\alpha}_{\eta}\hat{f}_k(t,\eta)|\langle l \rangle^{2s} e^{\lambda \langle l \rangle^s}|\rho_l|\langle k-l,\eta \rangle^{\sigma-s/2}\nonumber\\
    &\qquad \qquad \qquad \cdot e^{\lambda\langle k-l,\eta \rangle^s} |D^{\alpha}_{\eta}\hat{f}_{k-l}(t,\eta)|d\eta,\nonumber\\
    \lesssim & \lambda^2\Big(\|v^{\alpha}f\|_{\lambda,\sigma+s/2}^2\|\rho\|_{\lambda,d/2+2}+\|v^{\alpha}f\|_{\lambda,\sigma+s/2}\|\rho\|_{\lambda,\sigma+s/2}\|v^{\alpha}f\|_{\lambda,\sigma+d/2+3} \\
    &+\|v^{\alpha}f\|_{\lambda,\sigma}\|\rho\|_{\lambda,\sigma}\|v^{\alpha}f\|_{\lambda,d/2+3}+\|v^{\alpha}f\|_{\lambda,\sigma}^2\|\rho\|_{\lambda,d/2+3}\Big),\nonumber
\end{align}
where we have used that $3s/2-1\leq s/2$ in the first two terms. Moreover, since $\sigma\geq d/2+6$ 
\begin{equation}
    |A|\lesssim  \lambda^2\Big( \|v^{\alpha}f\|_{\lambda,\sigma+s/2}^2\|\rho\|_{\lambda,\sigma}+\|v^{\alpha}f\|_{\lambda,\sigma+s/2}\|\rho\|_{\lambda,\sigma+s/2}\|v^{\alpha}f\|_{\lambda,\sigma}+ \|v^{\alpha}f\|_{\lambda,\sigma}^2\|\rho\|_{\lambda,\sigma}\Big).
\end{equation} 

Now, we bound $|B|$ using Young's inequality
\begin{align}
    |B|&\lesssim \lambda \sum_{k,l}\int_{\R_{\eta}^d}\langle k,\eta \rangle^{\sigma+s/2} e^{\lambda \langle k,\eta \rangle^s}|D^{\alpha}_{\eta}\hat{f}_k(t,\eta)|\langle l \rangle|\rho_l|\langle k-l,\eta \rangle^{\sigma+3s/2-1} \\
    &\qquad \qquad \qquad \cdot e^{\lambda\langle k-l,\eta \rangle^s} |D^{\alpha}_{\eta}\hat{f}_{k-l}(t,\eta)|d\eta,\nonumber\\
    &\lesssim \lambda \sum_{k,l}\int_{\R_{\eta}^d}\langle k,\eta \rangle^{\sigma+s/2} e^{\lambda \langle k,\eta \rangle^s}|D^{\alpha}_{\eta}\hat{f}_k(t,\eta)|\langle l \rangle|\rho_l|\langle k-l,\eta \rangle^{\sigma+s/2} \\
    &\qquad \qquad \qquad \cdot e^{\lambda\langle k-l,\eta \rangle^s} |D^{\alpha}_{\eta}\hat{f}_{k-l}(t,\eta)|d\eta,\nonumber\\
    &\lesssim \lambda \|v^{\alpha}f\|_{\lambda,\sigma+s/2}^2\|\rho\|_{d/2+1}.
\end{align}
Moreover, since $\sigma\geq d/2+6$ 
\begin{equation}
    |B|\lesssim \lambda \|v^{\alpha}f\|_{\lambda,\sigma+s/2}^2\|\rho\|_{\sigma}.
\end{equation}
Now, we bound $|C|$, the last term in the decomposition for $E_{NL}^1$. By the mean value theorem and Young's inequality,
\begin{align}
    |C| \leq &\lambda \sum_{k,l}\int_{\R_{\eta}^d}\langle k,\eta \rangle^{\sigma+s/2}e^{\lambda \langle k,\eta \rangle^s}|D^{\alpha}_{\eta}\hat{f}_k(t,\eta)|\langle l \rangle^{s-1}|\rho_l| |\langle k,\eta \rangle^{\sigma-s/2}-\langle k-l,\eta \rangle^{\sigma-s/2}|\\
    &\qquad \qquad \qquad \cdot e^{\lambda\langle k-l,\eta \rangle^s}|\eta| |D^{\alpha}_{\eta}\hat{f}_{k-l}(t,\eta)|d\eta,\nonumber\\
    \lesssim &\lambda \sum_{k,l}\int_{\R_{\eta}^d}\langle k,\eta \rangle^{\sigma+s/2}e^{\lambda \langle k,\eta \rangle^s}|D^{\alpha}_{\eta}\hat{f}_k(t,\eta)|\langle l \rangle^{\sigma}|\rho_l|e^{\lambda\langle k-l,\eta \rangle^s}\langle k-l,\eta \rangle \\
    &\qquad \qquad \qquad \cdot |D^{\alpha}_{\eta}\hat{f}_{k-l}(t,\eta)|d\eta \nonumber
\end{align}
\begin{align}     
    &+ \lambda \sum_{k,l}\int_{\R_{\eta}^d}\langle k,\eta \rangle^{\sigma+s/2}e^{\lambda \langle k,\eta \rangle^s}|D^{\alpha}_{\eta}\hat{f}_k(t,\eta)|\langle l  \rangle^{\sigma}|\rho_l|\langle k-l,\eta \rangle^2 e^{\lambda\langle k-l,\eta \rangle^s}\nonumber\\
    &\qquad \qquad \qquad \cdot |D^{\alpha}_{\eta}\hat{f}_{k-l}(t,\eta)|d\eta, \nonumber\\
    &+ \lambda \sum_{k,l}\int_{\R_{\eta}^d}\langle k,\eta \rangle^{\sigma+s/2}e^{\lambda \langle k,\eta \rangle^s}|D^{\alpha}_{\eta}\hat{f}_k(t,\eta)|\langle l\rangle|\rho_l|\langle k-l,\eta \rangle^{\sigma}e^{\lambda\langle k-l,\eta \rangle^s}\nonumber\\
    &\qquad \qquad \qquad \cdot |D^{\alpha}_{\eta}\hat{f}_{k-l}(t,\eta)|d\eta,\nonumber\\ 
    \lesssim &  \lambda\Big(\|v^{\alpha}f\|_{\lambda,\sigma+s/2}\|\rho\|_{\sigma}\|v^{\alpha}f\|_{\lambda,d/2+2}+ \|v^{\alpha}f\|_{\lambda,\sigma+s/2}\|\rho\|_{\sigma}\|v^{\alpha}f\|_{\lambda,d/2+3} \\
    &+\|v^{\alpha}f\|_{\lambda,\sigma+s/2}\|\rho\|_{d/2+2}\|v^{\alpha}f\|_{\lambda,\sigma}\Big).\nonumber
\end{align}
Moreover, since $\sigma\geq d/2+6$ 
\begin{equation}
|C|\lesssim \lambda \|v^{\alpha}f\|_{\lambda,\sigma+s/2}\|\rho\|_{\sigma}\|v^{\alpha}f\|_{\lambda,\sigma}.
\end{equation}
Hence, by applying Lemma \ref{spatial_density} and adding up all terms in the previous decomposition of $|E_{NL}^1|$,  
\begin{align}
    |E_{NL}^1|\leq &|E_{NL}^{1,1}|+|E_{NL}^{1,2}|,\\
    \leq &|E_{NL}^{1,1}|+|A|+|B|+|C|,\\
    \lesssim &  \|f\|_{\lambda, \sigma+s/2,M}^2\Big(\lambda^2\|f\|_{\lambda,\sigma,M}+\lambda\|f\|_{\sigma,M}\Big)+\|f\|_{\lambda,\sigma,M}^2\|f\|_{\sigma,M}\Big(\lambda^2+1\Big).
\end{align}
\end{proof}
Finally, we estimate the second nonlinear term $E_{NL}^2$. 
\begin{claim}\label{second_nonlinear_term_gevrey}
\begin{equation}
    |E_{NL}^2|\lesssim \|v^{\alpha}f\|_{\lambda,\sigma}\|\rho\|_{\lambda,\sigma}\|v^{\alpha-j}f\|_{\lambda,\sigma}.
\end{equation}
\end{claim}

\begin{proof}
Since $e^{\lambda \langle k,\eta \rangle^s}\leq e^{\lambda\langle k-l,\eta \rangle^s}e^{\lambda \langle l \rangle^s}$ and by the mean value theorem,
\begin{align}
|E_{NL}^2|=&\sum_{|j|=1, j\leq \alpha}\sum_{l\in \Z^d_*}\sum_{k\in \Z^d}\int_{\R_{\eta}^d}|AD^{\alpha}_{\eta}\hat{f}_k(t,\eta)| \langle k,\eta \rangle^{\sigma}e^{\lambda\langle k,\eta \rangle^s}|\rho_l||\widehat{W}(l)||l_j| \\
    &\qquad \qquad \qquad \qquad \cdot  |D^{\alpha-j}_{\eta}\hat{f}_{k-l}(t,\eta)| d\eta,\nonumber\\
\lesssim&\sum_{j,k,l}\int_{\R_{\eta}^d}|AD^{\alpha}_{\eta}\hat{f}_k(t,\eta)| \langle k,\eta \rangle^{\sigma}e^{\lambda\langle k,\eta \rangle^s}|\rho_l|\langle l \rangle^{-1}  |D^{\alpha-j}_{\eta}\hat{f}_{k-l}(t,\eta)| d\eta,\\
\lesssim& \sum_{j,k,l}\int_{\R_{\eta}^d}|AD^{\alpha}_{\eta}\hat{f}_k(t,\eta)| \langle l \rangle^{-1}e^{\lambda\langle l \rangle^s}|\rho_l|\langle k-l,\eta \rangle^{\sigma}e^{\lambda\langle k-l,\eta \rangle^s} |D^{\alpha-j}_{\eta}\hat{f}_{k-l}(t,\eta)| d\eta \\
&+ \sum_{j,k,l}\int_{\R_{\eta}^d}|AD^{\alpha}_{\eta}\hat{f}_k(t,\eta)| \langle l \rangle^{\sigma-1}e^{\lambda\langle l \rangle^s}|\rho_l|e^{\lambda\langle k-l,\eta \rangle^s} |D^{\alpha-j}_{\eta}\hat{f}_{k-l}(t,\eta)| d\eta \nonumber
\end{align}
\begin{align}  
&+ \sum_{j,k,l}\int_{\R_{\eta}^d}|AD^{\alpha}_{\eta}\hat{f}_k(t,\eta)| \langle l \rangle^{\sigma-2}e^{\lambda\langle l \rangle^s}|\rho_l|\langle k-l,\eta \rangle e^{\lambda\langle k-l,\eta \rangle^s} |D^{\alpha-j}_{\eta}\hat{f}_{k-l}(t,\eta)| d\eta, \nonumber\\
&+ \sum_{j,k,l}\int_{\R_{\eta}^d}|AD^{\alpha}_{\eta}\hat{f}_k(t,\eta)| e^{\lambda\langle l \rangle^s}|\rho_l|\langle k-l,\eta \rangle^{\sigma-1}e^{\lambda\langle k-l,\eta \rangle^s} |D^{\alpha-j}_{\eta}\hat{f}_{k-l}(t,\eta)| d\eta,\nonumber \\
\lesssim & \|f\|_{\lambda, \sigma}\|\rho\|_{\lambda, d/2} \|v^{\alpha-j}f\|_{\lambda, \sigma}+\|f\|_{\lambda, \sigma}\|\rho\|_{\lambda, \sigma-1} \|v^{\alpha-j}f\|_{\lambda, d/2+1}\\
&+\|f\|_{\lambda, \sigma}\|\rho\|_{\lambda, \sigma-2} \|v^{\alpha-j}f\|_{\lambda, d/2+2}+\|f\|_{\lambda, \sigma}\|\rho\|_{\lambda, d/2+1} \|v^{\alpha-j}f\|_{\lambda, \sigma-1},\nonumber
\end{align}
by Young's inequality. Finally, since $\sigma\geq d/2+6$ 
\begin{equation}
    E_{NL}^2\lesssim \|v^{\alpha}f\|_{\lambda,\sigma}\|\rho\|_{\lambda,\sigma}\|v^{\alpha-j}f\|_{\lambda,\sigma}.
\end{equation}

\end{proof}
Finally, we conclude this section by putting all the previous estimates together in the proof of the main result in this paper. 
\subsection{Proof of Theorem \ref{propagation_regularity}}
\begin{proof}
By Claim \ref{linear_term_gevrey}, Claim \ref{first_nonlinear_term_gevrey} and Claim \ref{second_nonlinear_term_gevrey} 
\begin{align}
\dfrac{1}{2}\dfrac{d}{dt}\|f\|_{\lambda, \sigma,M}^2\lesssim & \|f\|_{\lambda,\sigma+s/2,M}^2\Big(\dot{\lambda}+\lambda^2\|f\|_{\lambda,\sigma,M}+\lambda \|f\|_{\sigma,M}+\lambda\Big)\\
&+\|f\|_{\lambda,\sigma,M}^2\Big[\|f\|_{\sigma,M}\Big(\lambda^2+1\Big)+1\Big].\nonumber
\end{align}
In particular, by choosing $\lambda$ such that $\dot{\lambda}+\lambda^2\|f\|_{\lambda,\sigma,M}+\lambda \|f\|_{\sigma,M}+\lambda\leq 0$, then 
\begin{equation}
\dfrac{1}{2}\dfrac{d}{dt}\|f\|_{\lambda,\sigma,M}^2\lesssim \|f\|_{\lambda,\sigma,M}^2\Big[\|f\|_{\sigma,M}\Big(\lambda^2+1\Big)+1\Big].    
\end{equation}
In order to obtain the estimate in Theorem \ref{propagation_regularity}, we use the quantitative estimate for the growth of the Sobolev norm $\|f\|_{\sigma,M}$ in Lemma \ref{sobolev_growth}. Thus, by Gronwall's lemma,
\begin{align}
\|f\|_{\lambda,\sigma,M}&\leq C\exp\Big(Ct+C\int_0^t \|f\|_{\sigma,M} ds \Big),\\ &\leq C\exp\Big\{Ct+Ct\exp\Big[C\int_0^t\Big(\|\nabla_x F\|_{\infty}+\| F\|_{\infty}+1\Big) ds\Big] \Big\}=:A(t), 
\end{align}
where $C$ is a constant depending on the initial data $f_0$, the dimension $d$, and the constants $M$ and $\sigma$. Furthermore, in the worse case scenario $\lambda<1$, so 
\begin{equation}
\dot{\lambda}+\lambda^2\|f\|_{\lambda,\sigma,M}+\lambda \|f\|_{\sigma,M}+\lambda\leq \dot{\lambda}+\lambda \Big(A(t)+1\Big),    
\end{equation}
which is clearly satisfied by $\lambda(t):= C\exp\Big[-\int_0^t\Big(2A(t)+1\Big)ds\Big].$ Thus, \begin{equation}
    \lambda(t)\geq C\exp\Big[-\int_0^t\Big(2A(t)+1\Big)ds\Big].
\end{equation}
\end{proof}







\section{Global existence of Gevrey solutions}\label{section_application}
In this section we prove Theorem \ref{global_existence_gevrey} as a direct consequence of Theorem \ref{propagation_regularity}. This applications uses crucially the global existence results discussed in the Introduction \ref{introduction}.
\begin{proof}[Proof of Theorem \ref{global_existence_gevrey}]
By the main theorem in \cite{BR}, there exists a unique global classical solution $f\in C(0,\infty;H^\sigma_{x,v}(\T^3\times \R^3))$ of the Vlasov-Poisson system (\ref{vlasov-poisson}) such that \begin{equation}\label{propagation_condition}
    \|F[f]\|_{W^{1,\infty}}(t)+\|\nabla_v f\|_{\infty,M}(t)<\infty,
\end{equation} for every $t\geq 0$. Since $\|f_0\|_{\sigma,M}\leq\|f_0\|_{\lambda_0,\sigma,M;s}<\infty$, the unique global classical solution satisfies $f\in C(0,\infty;H^\sigma_{x,v;M})$, by Lemma \ref{sobolev_growth}. As a result 
\begin{equation}
A(t):=C\exp\Big\{Ct+Ct\exp\Big[C\int_0^t\Big(\|F[f]\|_{W^{1,\infty}}+\|\nabla_v f\|_{\infty,M}+1\Big) ds\Big] \Big\}    
\end{equation} 
is finite. Therefore, the unique global classical solution $f\in C(0,\infty;H^\sigma_{x,v;M}(\T^3\times \R^3))$ of the Vlasov-Poisson system (\ref{vlasov-poisson}) satisfies 
\begin{equation}
    \|f\|_{\lambda,\sigma,M;s}(t)\leq A(t)<\infty.
\end{equation} In the case where the system is defined on $\R^3\times \R^3$, the result follows after replacing the theorem cited in \cite{BR} for the one in the theorem in page 1316 of \cite{S}.
\end{proof}

\end{document}